\documentclass[12pt,reqno]{amsart}

\usepackage[backref]{hyperref}

\usepackage{amsmath}
\usepackage{amssymb}
\usepackage{amsthm}

\usepackage{hyperref}
\usepackage{bm}

\usepackage{color}
\usepackage{enumerate}
\usepackage{tikz,lmodern}
\usepackage{amsfonts}
\usepackage{longtable}
\usepackage{supertabular}
\usepackage{multirow}
\usepackage{booktabs}

\newtheorem{theorem}{Theorem}[section]
\newtheorem{lemma}[theorem]{Lemma}
\newtheorem{corollary}[theorem]{Corollary}

\newtheorem{problem}[theorem]{Problem}

\newtheorem{definition}[theorem]{Definition}

\newtheorem*{fact*}{Fact}

\newtheorem*{claim*}{Claim}

\newcommand{\R}{\mathbb{R}}

\DeclareMathOperator{\Tr}{\mathrm{Tr}}

\begin{document}

\title[Steklov eigenvalues of graphs]{A transfer principle for Steklov eigenvalue estimates of graphs}

\author{Xiongfeng Zhan,\ Jin-Xin Zhou}

\address{School of mathematics and statistics\\
Beijing Jiaotong University\\
Beijing \\
100044, P. R. China}
\address{Beijing Key Laboratory of Biological Big Data and Topological Statistics\\
Beijing Jiaotong University\\
Beijing\\
100044, P.R. China}
\email{zhanxfmath@163.com (Xiongfeng Zhan), jxzhou@bjtu.edu.cn (Jin-Xin Zhou)}\date{}
\maketitle

\begin{abstract}
In this paper, we establish a new variant of the Burger-Brooks transfer principle, which allows us to apply spectral estimates for measured Riemannian surfaces to obtain the following result: There exists a universal constant $C>0$ such that, for every connected graph $G=(V, E)$ with boundary $B$, maximum degree $d_{\max}$ and genus $g$,
    \[\sigma_k(G, B)\leq C d_{\max}\frac{g+k}{|B|},\]
where $1\leq k\leq |B|$ and $\sigma_k(G, B)$ denotes the $k$-th Steklov eigenvalue of $G$ with boundary $B$. This bound is sharp up to a universal constant, thereby resolving a problem raised by Lin and Zhao [J. Lond. Math. Soc. (2) 112 (2025), Paper No. e70238]. Furthermore, when $B=V$, the above result yields an upper
bound for the Laplacian eigenvalues of graphs,  improving the previously known
bounds of Kelner, Lee, Price and Teng
[Geom. Funct. Anal. 21 (2011), 1117--1143]
and Amini and Cohen-Steiner
[Comment. Math. Helv. 93 (2018), 203--223].
\end{abstract}


\section{Introduction}

\subsection{The classical Steklov problem}
In 1902, Stekloff~\cite{Stekloff} considered the problem of liquid sloshing and introduced the Steklov eigenvalues for bounded domains in Euclidean spaces. Let $(M,\mathfrak{g})$ be a compact connected Riemannian manifold with nonempty boundary $\partial M$.
The classical Steklov problem asks for nonzero functions $f$ and real numbers $\sigma$ such that
\[\begin{cases}
	\Delta_{\mathfrak{g}} f=0 &  \text{in~} M, \\
	\partial_{\nu}f =\sigma f & \text{on~} \partial M,
  \end{cases}\]
where $\Delta_{\mathfrak{g}}$ denotes the Laplace Beltrami operator and $\partial_{\nu}f$ is the outward normal derivative of $f$ along the boundary. This problem can also be regarded as the eigenvalue problem of the Dirichlet-to-Neumann operator $\Lambda_{\mathfrak{g}}: H^\frac{1}{2}(\partial M)\rightarrow H^{-\frac{1}{2}}(\partial M)$, which is defined as
$$\Lambda_{\mathfrak{g}}\varphi:= \partial_{\nu}\hat{\varphi}|_{\partial M}, \forall \varphi\in H^\frac{1}{2}(\partial M),$$
where $\hat{\varphi}$ is the harmonic extension of $\varphi$ to $M$.
Since $\Lambda_{\mathfrak{g}}$ is a pseudo-differential operator, its spectrum is discrete and can be ordered as
\begin{align*}
	0 = \sigma_1(M,\mathfrak{g}) \leq \sigma_2(M,\mathfrak{g}) \leq \cdots
\end{align*}
These eigenvalues are called \emph{Steklov eigenvalues}, and $\sigma_2(M, \mathfrak{g})$ is called \emph{the first non-trivial
Steklov eigenvalue}. Different from the Dirichlet and Neumann problems, the eigenvalue $\sigma$ for the Steklov problem appears in the boundary condition. Over the past century, numerous papers have been devoted to the Steklov problem; see, for example, \cite{Bro01,Chen24,CG14,CGR18,Esc97,Esc99,FW17,FS19,GLS16,GP24,Kar17,KKP14,WX09,Wen54,YY17} and references therein.

\subsection{The graph Steklov problem}
A \textit{graph} $G$ is a pair of sets $(V, E)$, where $V$ is a finite set of elements called \textit{vertices}, and $E$ is a set of $2$-subsets of $V$ called \textit{edges}. The sets $V$ and $E$ are the \textit{vertex sets} and \textit{edge sets} of $G$, and are often denoted by $V(G)$ and $E(G)$, respectively. The \textit{order} of $G$ is the number of vertices of $G$. For convenience, an edge $e=\{u, v\}\in E$ is commonly written as $u\sim v$. The \textit{degree} $d_v$ of a vertex $v\in V$ is the number of edges incident to $v$. The maximum degree $d_{\max}$ of $G$ is $\max \{d_{v}:v\in V\}$. The \textit{(orientable) genus} $g$ of $G$ is the minimum genus of an (orientable) surface in which $G$ can be embedded. In this paper, all graphs are simple.

Let $G$ be a graph with vertex set $V$. The \textit{boundary} of $G$, denoted by $B(G)$, is chosen as an arbitrary nonempty subset of $V$, and we use $\Omega(G):=V\setminus B(G)$ to denote the set of rest vertices. Let $A$ be a finite nonempty set and let $\mathbb{R}^A$ denote the finite dimensional real vector space of functions from $A$ to $\mathbb{R}$, equipped with the standard inner product
\[\langle f, g\rangle=\sum_{x\in A}f(x)g(x).\]
For any $f\in \mathbb{R}^{V}$, the \textit{graph Laplacian $\Delta f$} is defined by
\[(\Delta f)(x)=\sum_{\{x,y\}\in E(G)}\bigl(f(x)-f(y)\bigr), \forall x\in V.\]
Note that $\Delta$ is a self-adjoint and positive semi-definite operator on $\mathbb{R}^{V}$.
The Steklov eigenvalue problem on $(G, B(G))$, introduced in \cite{HM20,HHW17}, is a discrete analogue of the classical Steklov problem. It aims to find nonzero functions $f\in \mathbb{R}^{V}$ and real numbers $\sigma\in \mathbb{R}$ such that
\begin{equation*}
	\begin{cases}
		\Delta f(x)=0, & x\in \Omega(G), \\
		\partial_vf(x)=\sigma f(x), & x\in B(G),
	\end{cases}
\end{equation*}
where $\partial_vf(x):=\sum\limits_{y\sim x} (f(x) - f(y))=(\Delta f)(x)$ for $x \in B(G)$. The real number $\sigma$ is called a \textit{Steklov eigenvalue} of the graph $G$ with boundary $B(G)$. Let $B=B(G)$. Then there are $|B|$ Steklov eigenvalues which can be ordered as
\[0=\sigma_1(G, B)\leq\sigma_2(G,B)\leq\cdots\leq\sigma_{|B|}(G, B),\]
where $\sigma_2(G, B)>0$ if and only if $G$ is connected.

Note that some authors (see, for example, \cite{LZ25a}) define $\partial_vf$ by
 \[\partial_v f(x):=\sum_{y\in\Omega(G),\,\{x,y\}\in E(G)} (f(x)-f(y)), \forall x\in B.\]
This difference does not affect the proof, because they only consider boundary sets $B$ that contain no edges of $G$. Our definition has an advantage, namely, it allows $B = V$, in which case the Steklov eigenvalues of $G$ with boundary $B$ coincide exactly with the Laplacian eigenvalues of $G$.

Following the pioneering work in \cite{HM20,HHW17}, the graph Steklov problem has been extensively studied in recent years by many authors, see for example, \cite{HH23,HH22b,HMW24,LZ25b,Per19,Per21,SY22,SY25,Tsc22,YY24,YY26}. In the extensive study of graph Steklov problem, much of the work has been focused on the upper bound estimates of the first nontrivial Steklov eigenvalue of graphs. For finite subgraphs in integer lattices, the upper bound of the first Steklov eigenvalue was given by Han and Hua in \cite{HH23}, where the authors also established a discrete analogue of Brock's result \cite{Bro01}. He and Hua \cite{HH22} obtained two types of bounds for the first Steklov eigenvalue of a tree $G$: $\sigma_2(G,B) \leq 4(d_{\max}-1)/|B|$ and $\sigma_2(G,B) \leq 2/L$, where $B$ and $L$ denote the boundary and diameter of the tree $G$, respectively. This was recently extended to block graphs by Lin and Zhao in \cite{LZ25a}, and moreover, Lin and Zhao in \cite{LZ25a} also obtained a beautiful result on the first Steklov eigenvalue of a planar graph $G$ with boundary $B$: $\sigma_2(G, B) \leq 8d_{\max}/|B|$. At the end of \cite{LZ25a}, Lin and Zhao pointed out that ``From geometric point of view, it is natural to ask what happens if we replace planar
graphs by genus $g$ graphs", and posed the following problem.

\begin{problem}[\rm{\cite[Problem 3.2]{LZ25a}}]\label{LZ-prob}
    What is the upper bound of Steklov eigenvalues for genus $g$ graphs with boundary $B$ when the degree is bounded above by $d_{\max}$?
\end{problem}

Let $G$ be a graph with boundary $B$ and genus $g$. The following partial results on the above problem have been proved:
\begin{itemize}
  \item $\sigma_2(G,B)\leq Cd_{\max}g^3/|B|$ in \cite{LLYZ26},
  \item $\sigma_2(G,B)\leq C(\mathrm{poly}(d_{\max}))g/|B|$ in \cite{CSZ26},
  \item $\sigma_k(G,B)\leq Cd_{\max}kg(\mathrm{log}~g)^2/|B| (1\leq k\leq |B|)$ in \cite{ZY26}.
\end{itemize}
In this paper, we shall give a solution of the above problem in general.

\subsection{The main results}

In spectral geometry, the Burger-Brooks transfer principle is a powerful method for comparing the spectrum of a continuous space, such as a Riemannian surface, with the spectrum of a graph obtained by discretizing that space. This principle was developed independently by Brooks \cite{Brooks86} and Burger \cite{Bur88}, and see \cite{BM04} for a random version of it. A remarkable application of the Burger-Brooks transfer principle to graph theory was developed by Amini and Cohen-Steiner \cite{AC18}. They established a graph-theoretic version of this principle and derived upper bounds for the normalized Laplacian eigenvalues of graphs in terms of their genus.

The first main result of this paper establishes a new variant of the Burger-Brooks transfer principle, which allows us to use the eigenvalues of measured Riemannian surfaces to estimate graph Steklov eigenvalues.

\begin{theorem}\label{thm:principle}
	There exist universal constants $C_1,C_2>0$ such that, for every connected graph $G$ with vertex set $V$ and genus $g$, for every nonempty subset $B\subseteq V$ and $1\leq k\leq |B|$, there exist:
    \begin{enumerate}[\rm (1)]
		\item  a closed smooth orientable Riemannian surface $(M,\mathfrak{g})$ of genus $g$ which depends only on $G$,
		\item and a finite nonzero Radon measure $\nu\ll dA_{\mathfrak{g}}$ depending on $G,B$ and $M$ with $\nu(M)=|B|$,
	\end{enumerate}
    satisfying that if $\lambda_k(M,\mathfrak{g},\nu)$, the $k$-th eigenvalue of $(M,\mathfrak{g},\nu)$, is bounded above by $C_1$, then $\sigma_k(G,B)\leq C_2d_{\max}\lambda_k(M,\mathfrak{g},\nu)$.
\end{theorem}
Our geometric construction and transfer mechanism are substantially different from those of Amini and Cohen-Steiner \cite{AC18}. Their method uses a two-fold cover of the surface and identifies the associated overlap operator with a weighted normalized graph Laplacian. In contrast, we thicken the vertices and edges of the original graph into vertex regions and edge strips and transfer the graph Dirichlet energy directly. Our construction is naturally adapted to the graph Steklov problem.
Combining it with Hassannezhad's upper bound \cite{Has11} on Riemannian surfaces, we obtain a general Steklov eigenvalue estimate for graphs in terms of their geometric genus.
\begin{theorem}\label{thm::main}
    There exists a universal constant $C>0$ such that, for every connected graph $G$ with boundary $B$ and genus $g$,
	\[\sigma_k(G,B)\leq C d_{\max}\frac{g+k}{|B|},\]
    where $1\leq k\leq |B|$, and $\sigma_k(G, B)$ is the $k$-th Steklov eigenvalue of the graph $G$ with boundary $B$.
\end{theorem}

The upper bound given in Theorem~\ref{thm::main} is sharp up to a constant factor $C$:
the linear dependence on the maximum degree is clearly optimal, as shown by star graphs; the $\frac{g+k}{|B|}$ dependence is also tight, as demonstrated by the examples in \cite[Remark 2.6]{AC18}. Theorem~\ref{thm::main} also extends the upper bound for the first nontrivial Steklov eigenvalues of planar graphs obtained by Lin and Zhao \cite{LZ25a} to all Steklov eigenvalues of graphs of genus $g$. In particular, it completely resolves Problem~\ref{LZ-prob}.

By taking $B=V$ in Theorem \ref{thm::main}, we obtain the following result which provides an upper bound for the Laplacian eigenvalues of graphs.

\begin{corollary}\label{cor1}
There exists a universal constant $C>0$ such that, for every connected graph $G$ of order $n$ and genus $g$,
\[\lambda_k(G)\leq C d_{\max}\frac{g+k}{n},\]
where $1\leq k\leq n$ and $\lambda_k(G)$ denotes the $k$-th Laplacian eigenvalue of $G$.
\end{corollary}

We note that Corollary~\ref{cor1} provides an improvement of a similar upper bound given by Kelner, Lee, Price and Teng in \cite{KLPT11}, where they proved that the $k$-th Laplacian eigenvalue of a graph $G$ of order $n$ and genus $g$ satisfies
\[
\lambda_k(G) \leq C d_{\max} \frac{(g+1)k\log^2(g+1)}{n},
\]
where $C$ is a universal constant.

Denote by $\lambda_k^{\rm{nr}}(G)$ the $k$-th normalized Laplacian eigenvalue of a graph $G$, and let $\delta=\max\{d_{\min}, 1\}$, where $d_{\min}$ is the minimum degree of $G$. Due to $\lambda_k^{\rm{nr}}(G)\leq\lambda_k(G)/\delta$, from Corollary~\ref{cor1} we immediately have
the following result, which improves Amini and Cohen-Steiner's bound for $\lambda_k^{\rm{nr}}(G)$ given in \cite[Theorem 1.1]{AC18} by the factor $1/\delta$.

\medskip
\begin{corollary}~\label{cor2}
There exists a universal constant $C>0$ such that, for every connected graph $G$ of order $n$ and genus $g$,
\[\lambda_k^{\rm{nr}}(G)\leq C \frac{d_{\max}}{\delta}\frac{g+k}{n},\]
where $1\leq k\leq n$.
\end{corollary}

\medskip
The remainder of the paper is organized as follows. In Section~\ref{section::2}, we recall the elementary facts needed for the proof of the main results. In Section~\ref{sec:3}, we construct a smooth conformal Riemannian metric on an oriented Euclidean polyhedral surface, and in Section~\ref{sec:4}, we construct Euclidean polyhedral surfaces from graphs and develop our transfer principle. In Section~\ref{sec:5}, we prove Theorem~\ref{thm:principle}, and in Section~\ref{sec:6}, we prove Theorem~\ref{thm::main}.

\section{Preliminaries}\label{section::2}

In this section, we review the basic facts that will be used in the paper.

\subsection{The Steklov eigenvalues of graphs}

Let $G=(V, E)$ be a connected graph with boundary $B$ and $f$ be a function in $\mathbb{R}^V$. Then
\begin{equation*}
	\langle\Delta f,f\rangle
	=\sum_{\{x,y\}\in E}\bigl(f(x)-f(y)\bigr)^2.
\end{equation*}
Given $\varphi\in \mathbb{R}^B$, the \textit{harmonic extension} of $\varphi$ is the unique function $\widehat{\varphi}\in\R^V$ \[\widehat{\varphi}|_B=\varphi \text{~and~} \Delta \widehat{\varphi}=0~\text{on }\Omega.\]
Equivalently,
\begin{equation}\label{eq::Detla hatf,f}
	\langle\Delta \widehat{\varphi},\widehat{\varphi}\rangle
	=\min\bigl\{\langle\Delta f,f\rangle:f\in \mathbb{R}^V,~f|_B=\varphi\bigr\}.
\end{equation}
The \textit{graph Dirichlet-to-Neumann operator} $\Lambda_{G,B}: \R^B\to\R^B$
is defined by
\[\Lambda_{G,B}\varphi=(\Delta\widehat{\varphi})|_B, \forall \varphi\in \R^B.\]

Let $\Omega=\Omega(G)$. With respect to the decomposition $V=B\cup \Omega$, write graph Laplacian $\Delta$ as the block matrix
\[\begin{pmatrix}
	 \Delta_{BB} & \Delta_{B\Omega} \\ \Delta_{\Omega B} & \Delta_{\Omega\Omega}
	 \end{pmatrix}.\]
Since $G$ is connected and $B\neq\emptyset$, the grounded Laplacian $\Delta_{\Omega\Omega}$ is positive definite. Therefore, we have
\begin{equation}\label{eq::DtN}
	\Lambda_{G,B}=\Delta_{BB}-\Delta_{B\Omega}\Delta_{\Omega\Omega}^{-1}\Delta_{\Omega B}.
\end{equation}
By definition, the eigenvalues of $\Lambda_{G,B}$ are precisely the Steklov eigenvalues of the graph $G$ with boundary $B$. Since $\Lambda_{G,B}$ is self-adjoint and positive semi-definite, the Courant-Fischer theorem gives
\begin{equation}\label{eq::minmax R}
	\sigma_k(G,B)
	=\min_{\substack{W\subseteq\R^B\\ \dim W=k}}
	\ \max_{\varphi\in W\setminus\{0\}}
	\frac{\langle\Delta \widehat{\varphi},\widehat{\varphi}\rangle}
	{\sum_{b\in B}\varphi(b)^2}.
\end{equation}

To end this subsection, we record the following elementary upper bound for the Steklov eigenvalues of graphs; see, for example, \cite[Lemma 2.1]{ZY26}.

\begin{lemma}\label{lem::2D}
    Let $G$ be a connected graph with boundary $B$ and maximum degree $d_{\max}$. Then, for every $1\leq k\leq |B|$, we have $\sigma_k(G,B)\leq 2d_{\max}$.
\end{lemma}

\subsection{Trace theorem and Poincar\'e  inequality}
We refer to \cite[Chapter 5]{Evans10} for the basic theory of Sobolev spaces, weak derivatives, trace operators, and Poincar\'e inequalities.
Let $\Omega\subset \mathbb{R}^2$ be a bounded connected Lipschitz domain. The \textit{Sobolev space} $H^1(\Omega)$ consists of all functions $f\in  L^2(\Omega)$ whose first-order weak derivatives belong to $L^2(\Omega)$. Denote by $dA$  the $2$-dimensional area measure on $\Omega$ and  $ds$ the $1$-dimensional arclength measure along the boundary $\partial\Omega$. The \textit{fractional Sobolev space} $H^{1/2}(\partial\Omega)$ consists of all functions $g\in  L^2(\partial\Omega)$ such that
\begin{equation*}
    \int_{\partial\Omega}\int_{\partial\Omega}\frac{|g(x)-g(y)|^2}{|x-y|^2}\,ds(x)ds(y)<\infty.
\end{equation*}

\begin{lemma}[Trace theorem]\label{lem::Trace}
    There exists a unique bounded linear operator
	    \[\operatorname{Tr}:H^1(\Omega)\longrightarrow H^{1/2}(\partial\Omega)\]
    such that $\Tr f=f|_{\partial\Omega}$ for every $f\in C^\infty(\overline{\Omega})$. In particular, there is a constant $C_{tr}(\Omega)>0$ such that
	\begin{equation*}\label{eq:: trace inequality}
		\int_{\partial\Omega}|\Tr f|^2\,ds\leq C_{tr}(\Omega)\left(\int_{\Omega}|f|^2\,dA+\int_{\Omega}|\nabla f|^2\,dA \right)
	\end{equation*}
    for every $f\in H^1(\Omega)$.
\end{lemma}

For $f\in H^1(\Omega)$, define the average of $f$ over $\Omega$ by
\[\overline{f}=\frac{1}{|\Omega|}\int_\Omega f\,dA,\]
where $|\Omega|=dA(\Omega)$.
Then the following Poincar\'e inequality holds.

\begin{lemma}[Poincar\'e inequality]\label{lem::P leq}
There exists a constant $C_P(\Omega)>0$ such that
	\begin{equation*}
		\int_{\Omega}|f-\bar{f}|^2\,dA\leq C_P(\Omega)\int_{\Omega}|\nabla f|^2\,dA
	\end{equation*}
for every $f\in H^1(\Omega)$.
\end{lemma}

The following lemma follows immediately from the trace theorem and the Poincar\'e inequality.

\begin{lemma}\label{lem::T leq}
    There exists a constant $C_T(\Omega)>0$ such that
    \[\int_{\Gamma}|\operatorname{Tr} f-\overline f|^2\,ds\leq C_T(\Omega)\int_{\Omega}|\nabla f|^2\,dA\]
    for every measurable $\Gamma\subset\partial\Omega$ and every $f\in H^1(\Omega)$.
\end{lemma}

\subsection{The measured Riemannian surface}

We briefly recall some basic notions concerning Riemann surfaces and measures. For further details, we refer the reader to \cite[Chapters 1 and 13]{Lee13}. A \textit{smooth surface} is a topological surface $X$ equipped with a smooth structure, namely, a maximal atlas of charts whose transition maps are smooth. A \textit{Riemannian surface} is a pair $(M, \mathfrak{g})$, where $M$ is a smooth surface and $\mathfrak{g}$ is a Riemannian metric on $M$, that is, smoothly varying inner product $\mathfrak{g}_p$ on the tangent spaces $T_pM$.

Let $M$ be a compact smooth surface equipped with its Borel $\sigma$-algebra. A locally finite Borel measure $\mu$ on $M$ is called a \textit{Radon measure} if it is inner and outer regular. Since $M$ itself is compact, every Radon measure on $M$ is finite.

Let $\mu$ and $\eta$ be two measures on $M$. We say that $\mu$ is \textit{absolutely continuous with respect to $\eta$}, denoted by $\mu\ll\eta$, if, for every Borel set $E\subseteq M$, $\eta(E)=0$ implies $\mu(E)=0$. For a nonnegative measurable function $\rho$, the notation $d\mu=\rho\,d\eta$ means that
\[\mu(E)=\int_E\rho\,d\eta\]
for every Borel set $E\subseteq M$. Such an identity immediately implies $\mu\ll\eta$.

Let $(M,\mathfrak{g})$ be a closed smooth Riemannian surface. Let $dA_{\mathfrak{g}}$ denote the Riemannian area measure induced by $\mathfrak{g}$ on $M$. Let $\nu$ be a finite nonzero Radon measure satisfying $\nu\ll dA_{\mathfrak{g}}$.
For $f\in C^\infty(M)$, define the Rayleigh quotient by
\[R_{M,{\mathfrak{g}},\nu}(f)=\frac{\int_M|\nabla_{\mathfrak{g}}f|_{\mathfrak{g}}^2\,dA_{\mathfrak{g}}}{\int_M f^2\,d\nu},\]
where $R_{M,{\mathfrak{g}},\nu}(f)=+\infty$ if $\int_M f^2\,d\nu=0$. The eigenvalues of the measured Riemannian surface $(M,\mathfrak{g},\nu)$ are defined by
\begin{equation}\label{eq::minmax R on surface}
	\lambda_k(M,\mathfrak{g},\nu)=\inf_{\substack{F\subset C^\infty(M)\\ \dim F=k}}\sup_{f\in F\setminus\{0\}} R_{M_{\mathfrak{g}}}(f,\nu),~~ k\geq1.
\end{equation}

The following result is the measured-metric version of the conformal eigenvalue estimate of Hassannezhad~\cite{Has11}; see also
Amini and Cohen-Steiner~\cite[p. 10, Eq. (5)]{AC18}, where it is explicitly observed that the proof of~\cite{Has11} extends to this setting.


\begin{theorem}\label{thm::eigenvalue of surface}
There exists a universal constant $A>0$ such that, for every closed smooth orientable Riemannian surface $(M,\mathfrak{g})$ of genus $g$, every finite nonzero Radon measure $\nu\ll dA_{\mathfrak{g}}$, and every $k\geq 1$,
    \begin{equation*}\label{eq::eigenvalue of surface}
\lambda_k(M,\mathfrak{g},\nu)\,\nu(M)\leq A(g+k).
    \end{equation*}
\end{theorem}

\subsection{Oriented Euclidean polyhedral surfaces}

We recall the standard Riemann surface structure on an arbitrary oriented Euclidean polyhedral surface; see \cite[Section~1.1]{Bob11}.

Let $X$ be a connected topological surface. A \textit{complex coordinate
chart} is a pair $(U,z)$, where $U\subseteq X$ is open and
$z:U\to z(U)\subseteq\mathbb C$ is a homeomorphism onto an open set.
A collection $\mathcal A$ of pairwise compatible complex coordinate
charts covering $X$ is called a \textit{complex atlas}, where two
charts $(U_i,z_i)$ and $(U_j,z_j)$ are \textit{compatible} if the transition map
\[z_j\circ z_i^{-1}:z_i(U_i\cap U_j)\longrightarrow z_j(U_i\cap U_j)\]
is holomorphic. Two complex atlases $\mathcal{A}$ and $\mathcal{B}$ are \textit{compatible} if $\mathcal{A}\cup\mathcal{B}$ is a complex atlas. This compatibility relation is an equivalence relation on the set of complex atlases.
A \textit{complex structure} on $X$ is a maximal
complex atlas. A \textit{Riemann surface} is a connected topological
surface equipped with a complex structure. Note that every Riemann surface is naturally a smooth surface, since holomorphic transition maps are smooth.

An \textit{oriented Euclidean polyhedral surface} is a compact surface obtained from finitely many disjoint Euclidean triangles by gluing their edges pairwise via Euclidean isometries, such that:
    \begin{enumerate}[{\rm (i)}]
    	\item Two distinct triangles meet only in a common vertex or a common edge;
        \item Every edge belongs to exactly two triangles;
        \item The triangles incident with any vertex occur in one cyclic sequence,
        with consecutive triangles sharing an edge through that vertex;
       \item The triangles can be oriented so that every edge gluing reverses the induced boundary orientations.
    \end{enumerate}

There are three local types of points on an oriented Euclidean polyhedral
surface:
\begin{enumerate}[{\rm (1)}]
    \item an interior point of a triangle;
    \item an interior point of a glued edge;
    \item a polyhedral vertex.
\end{enumerate}

\begin{figure}[htbp]
\centering
\begin{tikzpicture}[line width=.55pt,scale=.92]
\begin{scope}[xshift=0cm]
\fill[gray!35] (0,-0.0) circle (0.30);
\draw (0,-0.0) circle (0.30);
    \draw (-1.15,-.65)--(0,1.15)--(1.15,-.65)--cycle;
    \fill (0,0) circle (.045);
\end{scope}

\begin{scope}[xshift=4.25cm]

    \draw (-1.15,-.62)--(-.52,.95)--(.95,.65)--(.8,-.9)--cycle;
    \draw (-1.15,-.62)--(.95,.65);
    \fill[gray!35] (-0.02,0.02) circle (0.30);
    \draw (-0.02,0.02) circle (0.30);
    \fill (-.02,.03) circle (.045);

\end{scope}

\begin{scope}[xshift=8.5cm]
    \coordinate (O) at (0,0);
    \foreach \a/\b in {0/65,65/130,130/180,180/230,230/295,295/360}{
        \draw (O)--({1.12*cos(\a)},{1.12*sin(\a)});
        \draw ({1.12*cos(\a)},{1.12*sin(\a)})
            --({1.12*cos(\b)},{1.12*sin(\b)});
    }
    \fill[gray!35] (O) circle (0.30);
    \draw (O) circle (0.30);

    \fill (O) circle (.065);
    \node at (322.5:0.7) {$\theta_1$};
    \node at (23.5:0.7) {$\theta_2$};
    \node at (100:0.7) {$\cdots$};
    \node at (277.5:0.7) {$\theta_m$};
\end{scope}
\end{tikzpicture}
\caption{Three types of points on an oriented Euclidean polyhedral surface}
\end{figure}
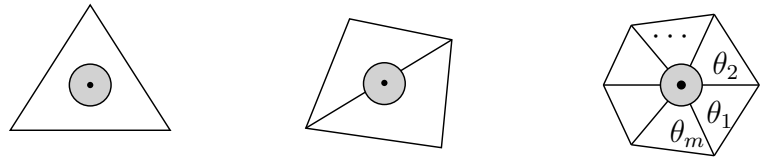

The induced intrinsic metric is locally Euclidean away from the finitely many polyhedral vertices. At a polyhedral vertex $p$, let $\theta_1, \theta_2, \ldots, \theta_m$ denote the angles at $p$ of the triangles incident to $p$. The sum
\[\Theta_p:=\theta_1+\theta_2+\ldots+\theta_m,\]
is called the \textit{total angle} at $p$.
If $\Theta_p\neq2\pi$, then $p$ is called a \textit{cone point}.

The following theorem shows that every oriented Euclidean polyhedral
surface carries a natural complex structure. We use the following slightly more precise version of
\cite[Theorem~3]{Bob11}. For completeness, we include a proof.

\begin{theorem}\label{thm::poly Riemann}
	Let $(X,h)$ be an oriented Euclidean polyhedral surface, and let $\Sigma$ be its finite set of cone points. Then $X$ admits a complex atlas $\{(U_i,z_i)\}_{i\in I}$ such that
    \begin{equation}\label{eq::poly Riemann}
        h|_{U_i\setminus\Sigma}=\eta_i |dz_i|^2,
    \end{equation}
    where $\eta_i$ is a smooth positive function on $z_i(U_i\setminus\Sigma)$ for every $i\in I$.
\end{theorem}

\begin{proof}
	We use the standard construction from \cite[Section 1.1.3]{Bob11}. At an interior point of a triangle or of a glued edge, complex coordinates are obtained from oriented Euclidean isometries.
	
	Let $p$ be a polyhedral vertex and $F_1,\ldots,F_m$ be the sequence of successive triangles incident to $p$. Let $\theta_1, \theta_2, \ldots, \theta_m$ be the angles at $p$ corresponding to $F_1,\ldots,F_m$. Then the total angle $\Theta_p=\theta_1+\theta_2+\ldots+\theta_m$. Set $\gamma_p=\frac{2\pi}{\Theta_p}$. Choose a sufficiently small $\varepsilon$-neighborhood $U_p$ of $p$, then $F_j\cap U_p$ is isometric to the Euclidean sector
	\begin{equation*}
		S_{j}(p)=\left\{\zeta=re^{i\varphi}:0\leq r<\varepsilon,\;\theta_1+\ldots+\theta_{j-1}\leq\varphi\leq\theta_1+\ldots+\theta_{j}\right\}.
	\end{equation*}
	For $re^{i\varphi}$ in each $S_{j}(p)$,  we define
	\begin{equation}\label{eq::z_p define}
		z_p=r^{\gamma_p}e^{i\gamma_p\varphi}.
	\end{equation}
	Since $\gamma_p\Theta_p=2\pi$, the boundary ray $\varphi=0$ of $S_{1}(p)$ and the boundary ray $\varphi=\Theta_p$ of $S_{m}(p)$ have the same image under this map. Therefore \eqref{eq::z_p define} induces a well-defined homeomorphism
	\[z_p:U_p\longrightarrow D_{\varepsilon^{\gamma_p}},~z_p(p)=0,\]
	where $D_{\varepsilon^{\gamma_p}}$ denotes the disk of radius $\varepsilon^{\gamma_p}$ in $\mathbb{C}$.
	
	On any simply connected open subset of $U_p$ not containing $p$, choose a holomorphic branch $L$ of $\mathrm{Log}~\zeta$. With respect to this branch, the cone coordinate can be written as $z_p=e^{\gamma_p L(\zeta)}=\zeta^{\gamma_p}$.
	Conversely, on any simply connected open subset of $D_{\varepsilon^{\gamma_p}}\setminus\{0\}$, the inverse relation is
	\begin{equation}\label{eq::inverse}
		\zeta=z_p^{1/\gamma_p},
	\end{equation}
	where the holomorphic branch of $z_p^{1/\gamma_p}$ is chosen so as to recover the corresponding part of the sector. Thus, the transition maps between the cone charts and the ordinary Euclidean charts are holomorphic. The transition maps between ordinary Euclidean charts are Euclidean isometries and hence are also holomorphic. Therefore, together with the ordinary Euclidean charts away from the cone points, the cone charts form a complex atlas $\{(U_i,z_i)\}_{i\in I}$ on $X$.
	
	It remains to express the metric $h$ in these complex coordinates.
	In a local developed Euclidean coordinate $\zeta$, we have $h=|d\zeta|^2$. Near a cone point $p$, this is exactly the inverse relation $\zeta=z^{1/\gamma_p}$ from \eqref{eq::inverse}. Otherwise, it is $\zeta=az+b$ for some $a,b\in \mathbb{C}$ as a Euclidean isometry. Hence
	\[ h=|\zeta'(z_i)|^2|dz_i|^2,\]
	where $\zeta'$ is smooth and nonzero in $z_i(U_i\setminus \Sigma)$ for any $(U_i,z_i)$, as desired.
\end{proof}

\section{Conformal metrics on Euclidean polyhedral surfaces}\label{sec:3}

In this section, we construct a smooth conformal Riemannian metric
on an oriented Euclidean polyhedral surface.

\begin{definition}[Conformal Riemannian metric]
	{\rm Let $X$ be a Riemann surface with a complex atlas $\{(U_i, z_i)\}_{i\in I}$. A \textit{conformal Riemannian metric} on $X$ is a Riemannian metric $\mathfrak{g}$ on $X$ such that, in every holomorphic coordinate $z_i$, \[\mathfrak{g}=\lambda_i(z_i)|dz_i|^2,\]
     where $\lambda_i$ is a smooth positive function on $z_i(U_i)$.}
\end{definition}

The condition that $\mathfrak{g}$ is conformal to the complex structure is independent of the choice of complex coordinates. Indeed, on an overlap $U_i\cap U_j$, we have
$z_j=F_{ji}(z_i)$, where $F_{ji}$ is biholomorphic. Therefore,
\[|dz_j|^2=|F_{ji}'(z_i)|^2|dz_i|^2,\]
where $|F_{ji}'(z_i)|^2$ is smooth and strictly positive. Thus
\[\mathfrak{g}=\frac{\lambda_i(z_i)}{|F_{ji}'(z_i)|^2}|dz_j|^2,\]
so the conformal factor remains smooth and strictly positive.

Although the polyhedral metric $h$ is not a smooth Riemannian metric at the cone points, we can replace $h$ by a smooth Riemannian metric conformal to the natural complex structure on $X$. The following standard result guarantees the existence of such a metric.

\begin{theorem}[\rm{\cite[p. 10]{Bob11}}]\label{thm::smooth conformal}
	Every compact Riemann surface admits a smooth conformal Riemannian metric.
\end{theorem}

Combining Theorem~\ref{thm::poly Riemann} with Theorem~\ref{thm::smooth conformal}, we obtain the following conformal replacement of the polyhedral metric.

\begin{lemma}\label{lem::surface replace}
	Let $(X,h)$ be an oriented Euclidean polyhedral surface, let $\Sigma\subset X$ be its finite set of cone points, and equip $X$ with the natural complex structure provided by Theorem~\ref{thm::poly Riemann}. Then there exist a smooth conformal Riemannian metric $\mathfrak{g}$ on $X$ and a positive smooth function
    \[\rho:X\setminus\Sigma\longrightarrow(0,\infty)\]
    such that
    \[h=\rho\mathfrak{g}~~\text{on }X\setminus\Sigma.\]
    Extending $\rho$ arbitrarily to $\Sigma$, one has
    \[dA_h=\rho\,dA_{\mathfrak{g}}\]
    as finite Radon measures on $X$. Moreover, for every $f\in C^\infty(X)$,
    \begin{equation}\label{eq::energy}
    	\int_X |\nabla_h f|_h^2\,dA_h=\int_X |\nabla_{\mathfrak{g}}f|_{\mathfrak{g}}^2\,dA_{\mathfrak{g}}.
    \end{equation}
\end{lemma}

\begin{proof}
	By Theorem \ref{thm::smooth conformal}, we can choose a smooth conformal Riemannian metric $\mathfrak{g}$ on $X$. Let $z:U\to z(U)\subset \mathbb{C}$ be a complex coordinate chart from Theorem \ref{thm::poly Riemann}. Then
    \[ h=\eta(z)|dz|^2,\]
    where $\eta$ is a smooth positive function on $z(U\setminus\Sigma)$. Since $\mathfrak{g}$ is conformal, in the same coordinate, we have $\mathfrak{g}=\lambda(z)|dz|^2$, where $\lambda$ is a smooth positive function on $z(U)$. Therefore
    \[h=\rho\mathfrak{g},~\rho=\frac{\eta(z)}{\lambda(z)}\]
    on $U\setminus\Sigma$. These local definitions agree on overlaps, since $\rho$ is uniquely determined by the pointwise relation $h=\rho\mathfrak{g}$. Thus $\rho$ is a well-defined smooth positive function on $X\setminus\Sigma$.

    Next, we prove the measure part of the lemma. Let $T_1,\ldots,T_N$ be all the Euclidean triangles of $(X,h)$, and for each $i$, let $\mathcal{L}^2|_{T_i}$ be the $2$-dimensional Lebesgue measure on $T_i$. For a Borel set $E\subset X$, let $E_i=E\cap T_i$. Then $E_i$ is a Borel subset of $T_i$ and
    \begin{equation}\label{eq::dA_h E}
    	dA_h(E)=\sum_{i=1}^N \mathcal{L}^2|_{T_i}(E_i).
    \end{equation}
    Since each $\mathcal L^2|_{T_j}$ is a finite Radon measure, $dA_h$ is also a finite Radon measure.

    Extend $\rho$ to $\Sigma$ by assigning arbitrary finite values. On $X\setminus\Sigma$, the relation $h=\rho \mathfrak{g}$ and the fact that $X$ is $2$-dimensional give \[dA_h=\rho\,dA_{\mathfrak{g}}.\]
    The finite set $\Sigma$ has zero area for both measures: this is standard for the smooth area measure, and it follows for $dA_h$ from \eqref{eq::dA_h E}. Consequently, for every Borel set $E\subseteq X$,
    \begin{equation*}
    	dA_h(E)=dA_h(E\setminus\Sigma)=\int_{E\setminus\Sigma}\rho\,dA_{\mathfrak{g}}
    	=\int_E\rho\,dA_{\mathfrak{g}}.
    \end{equation*}
    Thus $dA_h=\rho\,dA_{\mathfrak{g}}$ as Borel measures. Since $dA_h$ is a finite Radon measure, $\rho\,dA_{\mathfrak{g}}$ is the same finite Radon measure, and the identity immediately gives $dA_h\ll dA_{\mathfrak{g}}$.

    Finally, on $X\setminus\Sigma$, we have $h^{-1}=\rho^{-1}\mathfrak{g}^{-1}$ and hence
        \[|\nabla_hf|_h^2 =\rho^{-1}|\nabla_{\mathfrak{g}}f|_{\mathfrak{g}}^2.\]
    Multiplying by $dA_h = \rho\,dA_{\mathfrak{g}}$ and integrating it, we get \eqref{eq::energy}.
\end{proof}

\section{From graphs to surfaces}\label{sec:4}

In this section, we develop the geometric construction underlying
our transfer principle. We begin by constructing a Euclidean
polyhedral surface from a graph.

\subsection{Polyhedral thickenings of graphs}

A \textit{cellular embedding} of a connected graph $G$ in a closed connected orientable surface $S$ is an embedding  $G\hookrightarrow S$ such that every connected component of $S\setminus G$ is homeomorphic to an open disk. These components are called the \textit{faces} of the embedding.

\begin{lemma}[\rm{\cite[Proposition 3.4.1]{MT01}}]\label{lem::cellular}
    Every minimum  genus embedding of a connected  graph $G$ is cellular.
\end{lemma}

For a graph $G$, the Euclidean polyhedral surface constructed as
follows is called the \textit{polyhedral thickening} of $G$.

\begin{lemma}[Polyhedral thickening]\label{lem::poly thickening}
    Let $G=(V,E)$ be a connected graph of genus $g$ and maximum degree $d_{\max}>0$. Set
    \[w:=\frac{1}{10d_{\max}}.\]
    Then there exist a closed oriented Euclidean polyhedral surface $(M,h)$ of genus $g$, unit Euclidean squares $Q_v\subset M$ for $v\in V$, Euclidean rectangles $R_e\subset M$ for $e\in E$, and intervals $A_{v,e}\subset\partial Q_v$ for every incident pair $v\in e$, such that:
    \begin{enumerate}[{\rm (1)}]
    	\item the interiors of the squares $Q_v$ are pairwise disjoint, and each $Q_v$ is isometric to the unit Euclidean square
        \[([0,1]\times[0,1],dx^2+dy^2);\]

        \item for each $v\in V$, the intervals $A_{v,e}$, with $e\ni v$, are pairwise disjoint, each has length $w$, and none contains a corner of $Q_v$;

        \item if $e=\{u,v\}$, then $R_e$ is isometric to the Euclidean rectangle
        \[([0,1]\times[0,w],dx^2+dy^2),\]
        and its two opposite short sides are identified isometrically with $A_{u,e}$ and $A_{v,e}$;

        \item the interiors of the rectangles $R_e$ are pairwise disjoint, and each $R_e$ meets the vertex squares precisely along its two short sides.
    \end{enumerate}
\end{lemma}

\begin{proof}
    Choose a minimum-genus embedding
    \[\Phi:G\hookrightarrow S\]
    into a closed orientable surface $S$ of genus $g$. We first construct a piecewise Euclidean surface by thickening $G$ according to the embedding $\Phi(G)$.

    \medskip\noindent{\bf Stage 1.}\
    For each vertex $v\in V$, take an oriented Euclidean unit square $Q_v$. For each edge $e\in E$, take an oriented Euclidean rectangle
    \[R_e=[0,1]\times[0,w].\]
    Assume that all these pieces are pairwise disjoint.
    \medskip

    Fix a vertex $v$. The embedding $\Phi$, together with the orientation of $S$, determines a cyclic order on the edges incident with $v$. Hence we can choose pairwise disjoint closed intervals
    \[A_{v,e}\subset\partial Q_v,\qquad e\ni v,\]
    in this cyclic order, each of length $w$ and disjoint from the four corners of $Q_v$. This is possible since
    \[\sum_{e\ni v}|A_{v,e}|=d_vw\leq d_{\max}w=\frac{1}{10},\]
    which is less than the side length of  $Q_v$.

    \medskip\noindent{\bf Stage 2.}\
     Let $e=\{u,v\}$. Identify the short side $\{0\}\times[0,w]\subset R_e$ and $\{1\}\times[0,w]\subset R_e$ isometrically with $A_{u,e}$ and $A_{v,e}$, respectively. Choose the gluing maps to reverse the boundary orientations, so that
     the quotient inherits a consistent orientation. We obtain the quotient space
    \[N(G)=\left(\coprod_{v\in V}Q_v\ \sqcup\ \coprod_{e\in E}R_e\right)\bigg/\!\sim.\]


    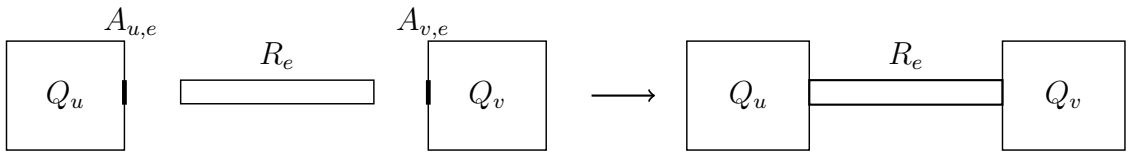
\begin{figure}[htbp]
    \centering
    \begin{tikzpicture}[line width=.55pt,scale=1.2]
		
		\draw (-6,-.55) rectangle (-4.7,.65);
		\node at (-5.35,.05) {$Q_u$};
		
		\draw[line width=1.8pt] (-4.7,-.05) -- (-4.7,.22);
		\node[above] at (-4.63,.55) {$A_{u,e}$};
		
		\draw (-4.08,-.04) rectangle (-1.95,.22);
		\node[above] at (-3.02,.22) {$R_e$};
		
		\draw[line width=1.8pt] (-1.35,-.05) -- (-1.35,.22);
		\node[above] at (-1.41,.55) {$A_{v,e}$};
		
		\draw (-1.35,-.55) rectangle (-.05,.65);
		\node at (-.7,.05) {$Q_v$};
		
		\draw[->,thick] (.45,.05) -- (1.15,.05);
		
		
		\draw (1.5,-.55) rectangle (2.85,.65);
		\node at (2.18,.05) {$Q_u$};
		
		\draw[thick] (2.85,-.05) rectangle (4.98,.22);
		\node[above] at (3.915,.22) {$R_e$};
		
		\draw (4.98,-.55) rectangle (6.33,.65);
		\node at (5.655,.05) {$Q_v$};
    \end{tikzpicture}
    \caption{The first two stages for one edge.}
    \label{fig:2}
    \end{figure}

    \noindent{\bf Stage 3.}
    Since the intervals around every $Q_v$ occur in the same cyclic order as the incident edges around $\Phi(v)$, the quotient $N(G)$ is homeomorphic to a regular neighborhood of $\Phi(G)$: the squares thicken the vertices and the rectangles thicken the edges. In particular, $N(G)$ is a compact orientable surface with boundary. By Lemma \ref{lem::cellular}, the embedding $\Phi$ is cellular, and hence the boundary components of $N(G)$ correspond precisely to the faces of the embedding $\Phi$.

    We now cap off each boundary component of $N(G)$ with a topological disk, thereby recovering a closed surface homeomorphic to $S$. Let $C$ be a boundary component of $N(G)$, which is a piecewise linear curve, and let its successive edge lengths be $\ell_1,\ldots,\ell_m$. Choose
    \[ R>\frac12\max_i\ell_i.\]
    For each $i$, take a Euclidean isosceles triangle with side lengths $R,R,\ell_i$. Glue the equal sides of these triangles cyclically, identifying all their apex vertices to a single interior vertex. The resulting abstract piecewise Euclidean disk has boundary isometric to $C$. Attach its boundary to $C$ by an isometry that reverses the induced boundary orientations. Repeating this operation for every component of $\partial N(G)$, we obtain a closed oriented surface $M$ with a piecewise Euclidean metric $h$. Since $M$ is homeomorphic to $S$, it has genus $g$.

    \begin{figure}[htbp]
    \centering
    \begin{tikzpicture}[line width=.6pt,scale=.95]
		\draw[dashed]
		(-3.7625,0.54992) -- (-2.8875,1.1199) -- (-1.95,0.6924) -- (-2.1375,-0.3765) -- (-3.2625,-0.6899) -- (-3.95,-0.0914) -- cycle;
		\node[align=center] at (-2.95,-1.445) {boundary component\\ $C\subset\partial N(G)$};
		
		\draw[->,thick] (-1.5,.18)--(-0.4,.18);
		\coordinate (O) at (1.72,.18);
		\coordinate (A) at (.72,.03);
		\coordinate (B) at (1.12,.92);
		\coordinate (C) at (2.05,1.18);
		\coordinate (D) at (2.72,.56);
		\coordinate (E) at (2.55,-.35);
		\coordinate (F) at (1.48,-.63);
		\draw[fill=gray!8] (A)--(B)--(C)--(D)--(E)--(F)--cycle;
		\foreach \P in {A,B,C,D,E,F}{\draw (O)--(\P);}
		\fill (O) circle (.04);
		\node[align=center] at (1.72,-1.445)
		{piecewise Euclidean disk\\glued along $C$};
	\end{tikzpicture}
	\caption{The third stage for a boundary component.}
	\label{fig:3}
    \end{figure}
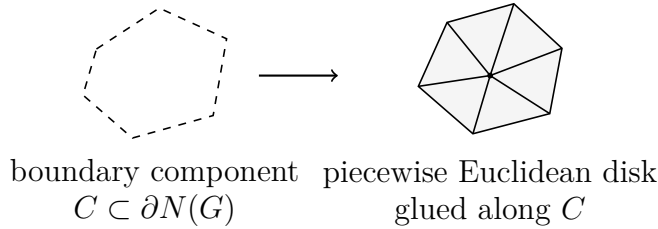

    The construction above gives the required metric properties and the topology. To obtain a Euclidean polyhedral surface in the precise sense of our definition, we subdivide each $Q_v$ into Euclidean triangles by taking all endpoints of the intervals $A_{v,e}$ as boundary vertices, and triangulate each $R_e$ compatibly with these subdivisions. This does not change the underlying metric or topology. Hence the resulting $(M,h)$ is a closed oriented Euclidean polyhedral surface satisfying all the stated properties. This completes the proof.
\end{proof}

\subsection{Local estimates}

In the remainder of this section, let $G=(V,E)$ be a fixed connected graph of genus $g$ and maximum degree $d_{\max}>0$, and let $M$ be the polyhedral thickening of $G$, equipped with the natural complex structure provided by Theorem~\ref{thm::poly Riemann}. For each
$v\in V$ and $e\in E$, let $Q_v$ and $R_e$ denote the Euclidean squares and rectangles, respectively, provided by Lemma~\ref{lem::poly thickening}.

\begin{lemma}\label{lem::Sobolev}
For every $v\in V$ and $e\in E$, the interiors $Q_v^\circ$ and $R_e^\circ$ are bounded Lipschitz domains. Moreover, for every $f\in C^\infty(M)$,
\[f|_{Q_v^\circ}\in H^1(Q_v^\circ)\quad\text{and}\quad f|_{R_e^\circ}\in H^1(R_e^\circ).\]
\end{lemma}

\begin{proof}
    All $Q_v$ and $R_e$ in $M$ are Euclidean polygons, and hence their interiors $Q_v^\circ$ and $R_e^\circ$ are Lipschitz domains. On these interiors the metric $h$ is Euclidean, $f$ is classically smooth and its classical first order derivatives are its weak derivatives.
	
	Since $M$ is compact, every $f\in C^\infty(M)$ is bounded. Moreover, the polyhedral area
	measure is finite, and therefore $\int_M |f|^2\,dA_h<\infty$. Hence
    \[f|_{Q_v^\circ}\in L^2(Q_v^\circ),~f|_{R_e^\circ}\in L^2(R_e^\circ).\]
	Applying Lemma \ref{lem::surface replace} to $(M,h)$, we have
    \begin{equation}\label{eq::nabla g=nabla h}
	    \int_M|\nabla_{\mathfrak{g}}f|_{\mathfrak{g}}^2\,dA_{\mathfrak{g}}
	    =\int_M|\nabla_hf|_h^2\,dA_h<\infty
    \end{equation}
    Restricting it to the interiors of the $Q_v$ and $R_e$, their first order weak derivatives also belong to  $L^2(Q_v^\circ)$ and $L^2(R_e^\circ)$, respectively.
\end{proof}

From now on, we write $Q_v$ and $R_e$ for the interiors $Q_v^\circ$ and $R_e^\circ$, respectively. Since their boundaries have zero $dA_h$-measure, this abuse of notation does not affect any of the integrals below. Let $f\in C^\infty(M)$. By Lemma~\ref{lem::Sobolev}, its restrictions to these $Q_v$ and $R_e$ belong to the corresponding Sobolev spaces. Hence the trace theorem and Poincar\'e inequality from Section~\ref{section::2} apply directly. Define the vertex average by
\begin{equation*}
	\phi_v(f):=\frac{1}{dA_h(Q_v)}\int_{Q_v}f\,dA_h=\int_{Q_v}f\,dA_h,
\end{equation*}
where we used $dA_h(Q_v)=1$. For each incident pair $v\in e$, define
the edge average
\begin{equation*}
p_{v,e}(f):=\frac{1}{w}\int_{A_{v,e}}\Tr f\,ds_h.
\end{equation*}

\begin{lemma}\label{lem::poincare and trace}
    Use the notation defined above. There exist universal constants $C_P,C_T>0$ such that
    \begin{equation}\label{eq::P leq}
    	\int_{Q_v}|f-\phi_v(f)|^2\,dA_h\leq C_P\int_{Q_v}|\nabla_h f|_h^2\,dA_h,
    \end{equation}
and
    \begin{equation}\label{eq::T leq}
        \int_{\partial Q_v}|\Tr f-\phi_v(f)|^2\,ds_h\leq C_T\int_{Q_v}|\nabla_h f|_h^2\,dA_h,
    \end{equation} for every $f\in C^\infty(M)$. The constants $C_P$ and $C_T$ are independent of $G$, $d_{\max}$ and $v$.
\end{lemma}

\begin{proof}
	Since each $Q_v$ is the Euclidean unit square and each $R_e$ is the $[0,1]\times [0,w]$ Euclidean rectangle, Lemma \ref{lem::P leq} and \ref{lem::T leq} yield
	\begin{equation*}
		\int_{Q_v}|f-\phi_v(f)|^2\,dA_h \leq C_P\int_{Q_v}|\nabla f|^2\,dA,
	\end{equation*}
	and
	\begin{equation*}
		\int_{\partial Q_v}|\Tr f-\phi_v(f)|^2\,ds_h\leq C_T\int_{Q_v}|\nabla f|^2\,dA
	\end{equation*}
	for some universal constants $C_P,C_T>0$ and every $f\in C^\infty(M)$. Since $h=dx^2+dy^2$ on $Q_v$, we have
	\begin{equation}\label{eq:: nabla=nabla h}
		|\nabla f|^2=|\nabla_hf|_h^2,~dA_h=dA
	\end{equation}
     on $Q_v$. The result follows.
\end{proof}

\begin{lemma}\label{lem::leq A_v,e}
	For every $v\in V$ and every $f\in C^\infty(M)$,
	\begin{equation*}
	    \sum_{e\ni v}w\bigl(\phi_v(f)-p_{v,e}(f)\bigr)^2
        \leq C_T\int_{Q_v}|\nabla_hf|_h^2\,dA_h,
    \end{equation*}
where $C_T$ is as in Lemma~\ref{lem::poincare and trace}.
\end{lemma}

\begin{proof}
	By Cauchy--Schwarz inequality, we have
	\[w\bigl(p_{v,e}(f)-\phi_v(f)\bigr)^2\leq\int_{A_{v,e}}|\Tr f-\phi_v(f)|^2\,ds_h.\]
    For fixed $v$, the intervals $A_{v,e}$, $v\in e$, are pairwise disjoint. So summing over $e\ni v$ and applying \eqref{eq::T leq}, we obtain the desired result.
\end{proof}

\begin{lemma}\label{lem::leq R_e}
    For any $e=\{u,v\}\in E$ and every $f\in C^\infty(M)$,
    \begin{equation*}
        w\bigl(p_{u,e}(f)-p_{v,e}(f)\bigr)^2\leq\int_{R_e}|\nabla_hf|_h^2\,dA_h.
    \end{equation*}
\end{lemma}

\begin{proof}
    Identify $R_e$ isometrically with the Euclidean rectangle $[0,1]\times[0,w]$ and put $F=f|_{R_e}$. For $x\in[0,1]$, define
    \[a(x)=\frac1w\int_0^w F(x,y)\,dy.\]
    Since $f$ is smooth on $M$, for every $x_0\in(0,1)$ the function $F$ is smooth in Euclidean coordinates near $\{x_0\}\times[0,w]$. Then $a(x)$ is smooth in a neighborhood of $x_0$. Hence $a\in C^\infty(0,1)$, and
    \begin{equation}\label{eq::a derivative}
        a'(x)=\frac1w\int_0^w\partial_xF(x,y)\,dy,~ x\in(0,1).
    \end{equation}
    Moreover, $F$ is continuous on the closed rectangle, because it is the restriction of the continuous function $f$ on $M$. Therefore $a$ is continuous on $[0,1]$ and
    \[a(0)=p_{u,e}(f),~a(1)=p_{v,e}(f).\]
    By Cauchy--Schwarz inequality in the $y$-variable and Fubini's theorem,
    \[\int_0^1|a(x)|^2\,dx\leq \frac1w\int_0^1\int_0^w|F(x,y)|^2\,dy\,dx<\infty,\]
    and using \eqref{eq::a derivative},
    \begin{equation}\label{eq::int a'}
    	\int_0^1|a'(x)|^2\,dx\leq \frac1w\int_0^1\int_0^w|\partial_xF(x,y)|^2\,dy\,dx<\infty.
    \end{equation}
    The last two integrals are finite because $F\in H^1((0,1)\times(0,w))$. Since $a\in C^1(0,1)$, its classical derivative is also its weak derivative. Consequently, $a\in H^1(0,1)$.

    The $1$-dimensional Sobolev fundamental theorem gives
    \[a(1)-a(0)=\int_0^1a'(x)\,dx.\]
    Therefore, Cauchy--Schwarz inequality on $(0,1)$ yields
    \[w\bigl|p_{u,e}(f)-p_{v,e}(f)\bigr|^2 =w\left|\int_0^1a'(x)\,dx\right|^2\leq w\int_0^1|a'(x)|^2\,dx\]
    From \eqref{eq::int a'}, we get
    \[w\bigl|p_{u,e}(f)-p_{v,e}(f)\bigr|^2 \leq\int_0^1\int_0^w|\partial_xF(x,y)|^2\,dy\,dx.\]
    Finally, because $h=dx^2+dy^2$ on $R_e$,
    \[|\partial_xF|^2\leq|\nabla F|^2=|\nabla_hf|_h^2,~dA_h=dy\,dx.\]
    The result follows.
\end{proof}

\subsection{Global estimates}

In this subsection, we establish some global estimates. We retain the
notation introduced in the previous subsection.

\begin{lemma}\label{lem::transfer1}
There is a universal constant $C_0>0$ such that
    \begin{equation*}
        \left\langle\Delta\bigl((\phi_v(f))_{v\in V}\bigr),(\phi_v(f))_{v\in V}\right\rangle\leq C_0d_{\max}\int_M|\nabla_{\mathfrak g}f|_{\mathfrak g}^2\,dA_{\mathfrak g}
    \end{equation*}
for every $f\in C^\infty(M)$.
\end{lemma}

\begin{proof}
	For convenience, we omit the variable $f$ in the proof.
    Let $\{u,v\}=e\in  E$. We write
    \[\phi_u-\phi_v=(\phi_u-p_{u,e})+(p_{u,e}-p_{v,e})+(p_{v,e}-\phi_v).\]
    Hence Cauchy--Schwarz inequality gives
    \[w(\phi_u-\phi_v)^2\leq 3w(\phi_u-p_{u,e})^2+3w(p_{u,e}-p_{v,e})^2+3w(p_{v,e}-\phi_v)^2.\]
    Sum over all edges and apply Lemma \ref{lem::leq A_v,e} once at every vertex and Lemma \ref{lem::leq R_e} once on every edge. The interiors of all vertex squares and edge rectangles are pairwise disjoint, so we have
    \[w\sum_{\{u,v\}\in E}(\phi_u-\phi_v)^2\leq C\int_M|\nabla_hf|_h^2\,dA_h=C\int_M|\nabla_{\mathfrak{g}}f|_{\mathfrak{g}}^2\,dA_{\mathfrak{g}}.\]
    Since $w^{-1}=10d_{\max}$, the result follows after changing the universal constant.
\end{proof}

Let $B$ be a nonempty subset of $V$, and set
\[\Omega_B=\bigcup_{b\in B}Q_b.\]
Define the measure $\nu$ by
\begin{equation}\label{eq::d nu}
    d\nu=\mathbf{1}_{\Omega_B}\,dA_h.
\end{equation}
By Lemma~\ref{lem::surface replace}, $d\nu=\mathbf{1}_{\Omega_B}\rho\,dA_{\mathfrak g}$.
Since $\nu$ is the restriction of the finite Radon measure $dA_h$ to the Borel set $\Omega_B$, it is itself a finite Radon measure. Moreover, $dA_h\ll dA_{\mathfrak g}$ implies $\nu\ll dA_{\mathfrak g}$. Finally, the sets $Q_b$, $b\in B$, have
pairwise disjoint interiors, each with $dA_h$-measure equal to $1$, and
their boundaries have zero $dA_h$-measure. Hence $\nu(M)=dA_h(\Omega_B)=|B|$.

\begin{lemma}\label{lem::transfer2}
    Every $f\in C^\infty(M)$ satisfies
    \begin{equation*}
    	\sum_{b\in B}\phi_b(f)^2\geq\int_Mf^2\,d\nu-C_P\int_M|\nabla_{\mathfrak{g}}f|_{\mathfrak{g}}^2\,dA_{\mathfrak{g}}.
    \end{equation*}
\end{lemma}

\begin{proof}
	Fix $b\in B$. Since $dA_h(Q_b)=1$, $\phi_b(f)=\int_{Q_b}f\,dA_h$. Hence
    \[\int_{Q_b}\bigl(f-\phi_b(f)\bigr)\,dA_h=\int_{Q_b}f\,dA_h-\phi_b(f)=0.\]
    Write $f=\phi_b(f)+(f-\phi_b(f))$,  then we have
    \begin{equation*}\label{eq::f^2 on Qb}
    	\begin{aligned}
    		\int_{Q_b}f^2\,dA_h=
    		\phi_b(f)^2+\int_{Q_b}|f-\phi_b(f)|^2\,dA_h.
    	\end{aligned}
    \end{equation*}
    Rearranging it and applying the Poincar\'e inequality \eqref{eq::P leq}, we get
    \[\phi_b(f)^2\geq \int_{Q_b}f^2\,dA_h-C_P\int_{Q_b}|\nabla_hf|_h^2\,dA_h.\]
    Summing this inequality over $b\in B$, we obtain
    \begin{equation}\label{eq::transfer2.1}
    	\sum_{b\in B}\phi_b(f)^2\geq \sum_{b\in B}\int_{Q_b}f^2\,dA_h-C_P\sum_{b\in B}\int_{Q_b}|\nabla_hf|_h^2\,dA_h.
    \end{equation}
    The interiors of the squares $Q_b$ are pairwise disjoint and their boundaries have zero $dA_h$-measure. By the definition $d\nu=\mathbf 1_{\Omega_B}dA_h$, we have
    \begin{equation}\label{eq::transfer2.2}
    	\sum_{b\in B}\int_{Q_b}f^2\,dA_h=\int_{\Omega_B}f^2\,dA_h=\int_Mf^2\,d\nu.
    \end{equation}
    The same disjointness gives
    \begin{equation}\label{eq::transfer2.3}
    	\sum_{b\in B}\int_{Q_b}|\nabla_hf|_h^2\,dA_h\leq\int_M|\nabla_hf|_h^2\,dA_h.
    \end{equation}
    By substituting \eqref{eq::transfer2.2} and \eqref{eq::transfer2.3} into \eqref{eq::transfer2.1}, we get
    \[\sum_{b\in B}\phi_b(f)^2\geq\int_Mf^2\,d\nu-C_P\int_M|\nabla_{h}f|_{h}^2\,dA_{h}.\]
    Using \eqref{eq::nabla g=nabla h}, we obtain the desired estimate.
\end{proof}

\section{The transfer principle}\label{sec:5}

In this section, we prove our transfer principle, namely
Theorem~\ref{thm:principle}. It is enough to prove the following
theorem.

Throughout this section, we retain the notation $C_P$ and $C_0$ for the universal
constants in Lemmas~\ref{lem::poincare and trace} and
\ref{lem::transfer1}, respectively. Let $M$ be the polyhedral thickening of a connected graph $G$,
equipped with the natural complex structure from
Theorem~\ref{thm::poly Riemann}. Let $\mathfrak{g}$ be the smooth
conformal Riemannian metric on $M$ from
Lemma~\ref{lem::surface replace}, let $B$ be a boundary of $G$ and let $\nu$ be the finite Radon
measure defined in \eqref{eq::d nu}.

\begin{theorem}\label{thm::transfer}
Use the notation given in the above paragraph. If $C_P\lambda_k(M,\mathfrak{g},\nu)<\frac12,$
then
$
\sigma_k(G,B)
\leq
2C_0d_{\max}\lambda_k(M,\mathfrak{g},\nu)$, where $1\leq k\leq |B|$.
\end{theorem}

\begin{proof}
	We suppose $d_{\max}>0$, since the case $d_{\max}=0$ is trivial.
	Choose $\varepsilon>0$ so that $C_P(\lambda_k(M,\mathfrak{g},\nu)+\varepsilon)<1/2$. By \eqref{eq::minmax R on surface}, there is a $k$-dimensional subspace $F\subset C^\infty(M)$ such that
    \begin{equation}\label{eq::test F}
    	\int_M|\nabla_{\mathfrak{g}}f|_{\mathfrak{g}}^2\,dA_{\mathfrak{g}}\leq(\lambda_k(M,\mathfrak{g},\nu)+\varepsilon)\int_Mf^2\,d\nu
    \end{equation}
    for every $f\in F$. Substituting \eqref{eq::test F} into Lemma \ref{lem::transfer2}, we get
    \begin{equation}\label{eq::phi_b geq}
    	\begin{aligned}
    		\sum_{b\in B}\phi_b(f)^2
    		&\geq\int_Mf^2\,d\nu-C_P\int_M|\nabla_{\mathfrak{g}}f|_{\mathfrak{g}}^2\,dA_{\mathfrak{g}}\\
            &\geq\bigl(1-C_P(\lambda_k(M,\mathfrak{g},\nu)+\varepsilon)\bigr)\int_Mf^2\,d\nu\\
            &\geq\frac12\int_Mf^2\,d\nu.
        \end{aligned}
    \end{equation}
    Define
    \[T_B:F\longrightarrow\R^B,~T_Bf=(\phi_b(f))_{b\in B}.\]
    We claim that $T_B$ is injective. If $T_Bf=0$, then \eqref{eq::phi_b geq} yields $\int_Mf^2\,d\nu=0$, and from \eqref{eq::test F} we know that $\int_M|\nabla_{\mathfrak{g}}f|_{\mathfrak{g}}^2\,dA_{\mathfrak{g}}=0$. Hence $f$ is constant on the connected smooth surface $M$. Since $\nu(M)>0$ and the $L^2(\nu)$ norm of $f$ is zero, the constant is zero. Thus $f=0$.

    Therefore $T_B(F)$ has dimension $k$. The full vector $(\phi_v(f))_{v\in V}$ extends $T_Bf$ from $B$ to $V$. Then by
    \eqref{eq::Detla hatf,f}, we have
    \[\langle\Delta \widehat{T_Bf},\widehat{T_Bf}\rangle\leq\left\langle\Delta \bigl((\phi_v(f))_{v\in V}\bigr),(\phi_v(f))_{v\in V}\right\rangle.\]
    Thus Lemma \ref{lem::transfer1} gives
    \[\langle\Delta \widehat{T_Bf},\widehat{T_Bf}\rangle\leq C_0d_{\max}\int_M|\nabla_{\mathfrak{g}}f|_{\mathfrak{g}}^2\,dA_{\mathfrak{g}}.\]
    Since $T_Bf=(\phi_b(f))_{b\in B}$, we get
    \[\frac{\langle\Delta \widehat{T_Bf},\widehat{T_Bf}\rangle}{\sum_{b\in B}(T_Bf)(b)^2}\leq \frac{C_0d_{\max}\int_M|\nabla_{\mathfrak{g}}f|_{\mathfrak{g}}^2\,dA_{\mathfrak{g}}}{\sum_{b\in B}\phi_b(f)^2}.\]
    Substituting \eqref{eq::test F} and \eqref{eq::phi_b geq} into the right hand side of the above inequality, we get
    \[\frac{\langle\Delta \widehat{T_Bf},\widehat{T_Bf}\rangle}{\sum_{b\in B}(T_Bf)(b)^2}\leq 2C_0d_{\max}(\lambda_k(M,\mathfrak{g},\nu)+\varepsilon).\]
    Applying \eqref{eq::minmax R} to $T_B(F)$ and taking $\varepsilon\to0$, we obtain the  desired estimate.
\end{proof}

\section{Proof of Theorem \ref{thm::main}}\label{sec:6}

\begin{proof}[Proof of Theorem~\ref{thm::main}]
	The case $d_{\max}=0$ is trivial. We assume $d_{\max}>0$. In what follows, we retain the notation $C_P$ and $C_0$ for the universal
constants in Lemmas~\ref{lem::poincare and trace} and
\ref{lem::transfer1}, respectively.
	Fix $1\leq k\leq|B|$. Since $\nu$ is the finite Radon measure defined in \eqref{eq::d nu} which satisfies $\nu\ll dA_{\mathfrak{g}}$ and $\nu(M)=|B|$, by Theorem \ref{thm::eigenvalue of surface}, we have
    \begin{equation}\label{eq::lambda_k}
    	\lambda_k(M,\mathfrak{g},\nu)\leq A\frac{g+k}{|B|}.
    \end{equation}
    Suppose first that $\frac{g+k}{|B|}\leq(3AC_P)^{-1}$.
    Then \eqref{eq::lambda_k} gives $C_P\lambda_k(M,\mathfrak{g},\nu)\leq1/3<1/2$. Theorem \ref{thm::transfer} yields
    \[\sigma_k(G,B)\leq2C_0d_{\max}\lambda_k(M,\mathfrak{g},\nu)\leq2C_0Ad_{\max}\frac{g+k}{|B|}.\]
    In the remaining range, $\frac{g+k}{|B|}>(3AC_P)^{-1}$. Using Lemma \ref{lem::2D}, we get
    \[\sigma_k(G,B)\leq2d_{\max}<6AC_Pd_{\max}\frac{g+k}{|B|}.\]
    Setting $C=\max\{2C_0A,6AC_P\}$, we obtain the desired universal constant.
\end{proof}

\section*{Acknowledgements}
This work was supported by the National Natural Science Foundation of China (Grants Nos. 12425111, 12331013).

\subsection*{Declaration on the Use of AI}
The key ideas and methods of this paper were developed by the authors, inspired by the previous work \cite{ZY26} and the approach of \cite{AC18}. AI tools were only employed to assist in simplifying the proof, check for errors and draw figures in this paper. All mathematical claims, proofs, and citations were independently verified by the authors, who assume full responsibility for any errors.

\end{document}